\documentclass[reqno,a4paper]{amsart}

\usepackage{amsmath}
\usepackage{amssymb}
\usepackage{amsthm}
\usepackage{enumerate}

\theoremstyle{plain}
\newtheorem{theorem}{Theorem}[section]
\newtheorem{proposition}[theorem]{Proposition}
\newtheorem{lemma}[theorem]{Lemma}

\theoremstyle{definition}
\newtheorem{definition}{Definition}[section]

\theoremstyle{remark}
\newtheorem{example}{Example}
\numberwithin{equation}{section}

\begin{document}

\title[On strongly separately continuous functions]%
{On strongly separately continuous functions\\ on sequence spaces}
\author[Olena Karlova \and Tom\'{a}\v{s} Visnyai]%
{Olena Karlova* \and Tom\'{a}\v{s} Visnyai**}

\newcommand{\acr}{\newline\indent}

\address{\llap{*\,} Chernivtsi National University,\acr
                   Faculty of Mathematics and Informatics,\acr
                   Department of Mathematical Analysis,\acr 
                   Kotsyubyns'koho 2, 58 012 Chernivtsi,\acr
                   Ukraine}
\email{maslenizza.ua@gmail.com}

\address{\llap{**\,} Institute of Information Engineering, Automation and Mathematics, \acr
                   Faculty of Chemical and Food Technology, \acr
                   Slovak University of Technology in Bratislava,\acr
                   Radlinskeho 9, 812 37 Bratislava,\acr
                   Slovak Republic}
\email{tomas.visnyai@stuba.sk}


\begin{abstract}
  We study strongly separately continuous real-valued function defined on the Banach spaces $\ell_p$.  Determining sets for the class of strongly separately continuous functions on $\ell_p$ are characterized. We prove that for every $1\le \alpha<\omega_1$ there exists a strongly separately continuous function which belongs the $(\alpha+1)$'th Baire class and does not belong to the $\alpha$'th Baire class on $\ell_p$. We show that any open set in $\ell_p$ is the set of discontinuities of a strongly separately continuous real-valued function.
\end{abstract}

\maketitle

\section{Introduction}
Let $(X_t:t\in T)$ be a family of sets $X_t$ with $|X_t|>1$ for all $t\in T$. For $S\subseteq T$ we put $X_S=\prod\limits_{t\in S}X_t$. If $S\subseteq S_1\subseteq T$, $x=(x_t)_{t\in T}\in X_T$, $a=(a_t)_{t\in S_1}\in X_{S_1}$, then we denote by $x_S^a$ the point $(y_t)_{t\in T}\in X_T$ defined by
$$
y_t=\left\{\begin{array}{ll}
             a_t, & t\in S, \\
             x_t, & t\in T\setminus S.
           \end{array}
\right.
$$
In the case  of $S=\{s\}$ we write $x_s^a$ instead of $x_{\{s\}}^a$.

If $S\subseteq T$, then we put $\pi_S(x)=(x_t)_{t\in S}$.

For $n\in\mathbb N$ let $\sigma_n(x)=\{y=(y_t)_{t\in T}\in X_T: |\{t\in T:y_t\ne x_t\}|\le n\}$ and $\sigma(x)=\bigcup\limits_{n=1}^\infty \sigma_n(x)$. Each of the sets of the form $\sigma(x)$ we call {\it a $\sigma$-product of $X_T$}.

A set $X\subseteq X_T$ is called {\it $\mathcal S$-open}~\cite{KaMykh:MasloSSC} if $\sigma_1(x)\subseteq X$ for all $x\in X$.
Notice that the definition of an $\mathcal S$-open set develops the definition of a set of the type ($P_1$) introduced in \cite{CSV}. Observe that $\sigma$-products of two distinct points of $X_T$ either coincide, or do not intersect. Thus, the family of all $\sigma$-products of an arbitrary $\mathcal S$-open set $X\subseteq X_T$ generates a partition of $X$ on  mutually disjoint $\mathcal S$-open sets, which we will call {\it $\mathcal S$-components of $X$.}


Let $X\subseteq X_T$ be an $\mathcal S$-open set, $\tau$ be a topology on $X$ and let $(Y,d)$ be a metric space. A mapping $f:(X,\tau)\to Y$ is called {\it strongly separately continuous at a point $a\in X$ with respect to the $t$-th variable} if
  $$
\lim_{x\to a}d(f(x),f(x_t^a))=0.
  $$
A mapping $f:X\to Y$ is {\it strongly separately continuous at a point $a\in X$} if $f$ is strongly separately continuous at  $a$ with respect to every variable $t\in T$; and $f$ is {\it strongly separately continuous on the set $X$} if $f$ is strongly separately continuous  at every point $a\in X$ with respect to every variable $t\in T$. Strongly separately continuous functions we will also call {\it ssc-functions} for short.

The notion of real-valued ssc-function defined on $\mathbb R^n$ was introduced by  Omar Dzagnidze~\cite{Dzagnidze}, who proved
that a function $f:\mathbb R^n\to\mathbb R$ is strongly separately continuous at $x^0$ if and only if $f$ is continuous at $x^0$.

In \cite{CSV} the authors extended the notion of the strong separate continuity to functions defined on the Hilbert space $\ell_2$ equipped with the  norm topology. They proved that there exists a real-valued ssc-function on $\ell_2$ which is everywhere discontinuous. Determining sets (see Definition~\ref{def:det_set}) for for the class of all ssc-functions were also studied in~\cite{CSV}.

The second named author continued to study properties of strongly separately continuous functions on $\ell_2$ in  \cite{TV} and constructed a strongly separately continuous function $f:\ell_2\to\mathbb R$ which belongs to the third Baire class and is not quasi-continuous at every point. Moreover, he gave a sufficient condition for a strongly separately continuous function to be continuous on $\ell_2$ and a sufficient condition for a subset of $\ell_2$ to be determining  in the class of real-valued ssc-functions.

The first named author extended the concept of an ssc-function on any $\mathcal S$-open subset of a product of topological spaces \cite{KaMykh:MasloSSC}. A characterization of the set of all points of discontinuity of strongly separately continuous functions defined on a $\sigma$-product of a sequence of finite-dimensional normed spaces was given in \cite{KaMykh:MasloSSC}. Further, the Baire classification of ssc-functions defined on the space $\mathbb R^\omega$ equipped with the topology of pointwise convergence was investigated in~\cite{KaRAEX:2015}. Moreover, it was shown in \cite{KaRAEX:2015} that if  $X$ is a product of normed spaces and $a\in X$ then for any open set $G\subseteq \sigma(a)$ there is a strongly separately continuous function $f:\sigma(a)\to \mathbb R$ such that the discontinuity point set of $f$ is equal to~$G$.
Strongly separately continuous functions defined on a box-product of topological spaces were considered in~\cite{KaMatStud:2015}.

Here we study ssc-functions defined on the spaces $\ell_p$ with $1\le p<+\infty$ of all sequences $(x_n)_{n=1}^\infty$ of real numbers for which the series $\sum\limits_{n=1}^\infty |x_n|^p$ is convergent. In Section~\ref{sec:det_sets} we find a necessary and sufficient condition on a subset of $\ell_p$ to be determining in the class of all real-valued ssc-functions on $\ell_p$. In the third section we show that for any ordinal $\alpha\in [1,\omega_1)$ there exists an $\mathcal S$-open set $E\subseteq\ell_p$ which is of the $\alpha$'th additive Borel class and does not belong to the $\alpha$'th multiplicative Borel class. Using this fact we construct a real-valued ssc-function on $\ell_p$ which belongs to the $(\alpha+1)$'th Baire class and does not belong to the $\alpha$'th Baire class.  In Section~\ref{sec:discont} we prove that for any open nonempty set $G\subseteq \ell_p$ and  $1\le p<\infty$ there exists a strongly separately continuous function $f:\ell_p\to \mathbb R$ which is discontinuous exactly on $G$.

\section{Determining sets for  strongly separately continuous functions}\label{sec:det_sets}
Let $(X,Y)$ be a pair of topological spaces and $\mathcal F(X,Y)$ be a class of mappings between $X$ and $Y$.
\begin{definition}\label{def:det_set}
  {\rm A set $E\subseteq X$ is called {\it determining for the class $\mathcal F(X,Y)$} if for any mappings $f,g\in\mathcal F(X,Y)$ the equality  $f|_E=g|_E$ implies that $f=g$ on $X$.}
\end{definition}

It is well-known that any everywhere dense subset $E$ of a topological space  $X$ is determining in the class $C(X,\mathbb R)$ of all continuous real-valued functions on $X$. The theorem of Sierpi\'{n}ski~\cite{Sierp} tells us that any everywhere dense subset $E\subseteq\mathbb R^2$ is determining in the class $CC(\mathbb R^2,\mathbb R)$   of all separately continuous real-valued functions on $\mathbb R^2$.

In this section we give necessary and sufficient conditions on a subset of an $\mathcal S$-open set $X\subseteq\prod\limits_{n=1}^\infty X_n$ to be determining for the class ${\rm SSC}(X)$ of all strongly separately continuous real-valued functions on $X$.

The following two notions were introduced in \cite{KaMykh:MasloSSC}.

\begin{definition}
  {\rm A set $A\subseteq X_T$ is called {\it projectively symmetric with respect to a point $a\in A$} if $x_{t}^a\in A$ for all $t\in T$
  and $x\in A$.}
\end{definition}

\begin{definition}
  {\rm Let $X\subseteq X_T$ and $\tau$ be a topology on $X$. Then $(X,\tau)$ is said to be {\it locally projectively symmetric} if every $x\in X$ has a base of projectively symmetric neighborhoods with respect to $x$.}
\end{definition}

\begin{definition}
  {\rm Let $(X_t:t\in T)$ be a family of topological spaces and $X\subseteq X_T$ be an $\mathcal S$-open set. We say that a topology $\tau$ on $X$ is
  {\it finitely generated} if for every $a\in X$ and every finite set $S\subseteq T$ the space $(X_S\times\prod\limits_{t\in T\setminus S}\{a_t\},\tau)$ is homeomorphic to the space $X_S$ with the topology of pointwise convergence.}
\end{definition}

Notice that an arbitrary $\mathcal S$-open subset of a product $X_T$ of topological spaces $X_t$ equipped  with the  topology of pointwise convergence is a locally projectively symmetric space with a finitely generated topology. All classical spaces of sequences as the space $c$ of all convergence sequences or the spaces $\ell_p$ with $0<p\le \infty$ are locally projectively symmetric with a finitely generated topology.

Throughout the paper we consider only finitely generated topologies.

We need the following result~\cite[Theorem 3.4]{KaMykh:MasloSSC}.

\begin{theorem}\label{th:operations-ssc}
Let  $X\subseteq X_T$ be an $\mathcal S$-open set equipped with a locally projectively symmetric topology $\tau$ and $x_0\in X$.
  \begin{enumerate}
    \item[(i)]   If $f:X\to\mathbb R$ and $g:X\to \mathbb R$ are ssc-functions at $x_0$, then
    $f(x)\pm g(x)$, $f(x)\cdot g(x)$, $|f(x)|$, $\min\{f(x),g(x)\}$ and $\max\{f(x),g(x)\}$ are ssc-functions at $x_0$.

    \item[(ii)] If $(f_n)_{n=1}^\infty$ is a sequence of  ssc-functions at the point $x_0$ and the series $\sum\limits_{n=1}^\infty f_n(x)$ is convergent uniformly on $X$, then the sum $f(x)$ is an ssc-function at $x_0$.

    \item[(iii)] If $(f_i)_{i\in I}$ is a locally finite family of ssc-functions  $f_i:X\to \mathbb R$ at $x_0$, then $f(x)=\sum\limits_{i\in I}f_i(x)$ is an ssc-function at $x_0$.
  \end{enumerate}
\end{theorem}

It is worth noting that the quotient of two real-valued ssc-functions need not be an ssc-function~\cite{Karlova:BMZh:2014:SSC}.

\begin{definition}
{\rm Let $(X_t:t\in T)$ be a family of topological spaces $X_t$ and $a\in X_T$. A set $W\subseteq \sigma(a)$ is called {\it nearly open in $\sigma(a)$} if for any finite set $T_0\subseteq T$ the set
  $$
  W_{T_0}=\{x\in X_{T_0}: a_{T_0}^x\in W\}
  $$
is open in the space $X_{T_0}$ equipped with the topology of pointwise convergence.}
\end{definition}

\begin{definition}
{\rm Let $(X_t:t\in T)$ be a family of topological spaces $X_t$, $X\subseteq X_T$ be an $\mathcal S$-open set and $(\sigma_i:i\in I)$ be a partition of $X$ on $\mathcal S$-components. A set $W\subseteq X$ is called {\it nearly open in $X$} if $W\cap \sigma_i$ is nearly open in $\sigma_i$ for every $i\in I$. }
\end{definition}

\begin{definition}
  {\rm Let $X\subseteq X_T$ be an $\mathcal S$-open set. A set $H\subseteq X$ is called {\it nearly closed  in $X$} if the complement $X\setminus H$ is nearly open in $X$.}
\end{definition}

For a set $A\subseteq\sigma(a)$ we put
$$
\overline{A}^{\bullet}=\{x\in\sigma(a):  W\cap A\ne\emptyset \,\,\,\,\forall W - \mbox{\,\,nearly open}, x\in W\}.
$$
\begin{definition}
  {\rm A set $A\subseteq\sigma(a)$ is said to be {\it  super dense (p-dense)} in $\sigma(a)$ if $\overline{A}^{\bullet}=\sigma(a)$ (respectively, $\overline{A}^{p}=\sigma(a)$, where $\overline{A}^p$ stands for the closure of $A$ in the topology of pointwise conergence on $\sigma(a)$).}
\end{definition}

\begin{definition}
{\rm A subset $A$ of an $\mathcal S$-open set  $X=\bigcup\limits_{i\in I}\sigma_i\subseteq X_T$, where $(\sigma_i:i\in I)$ is a family of all $\mathcal S$-components of $X$, is called {\it super dense} in $X$ if for every $i\in I$ the set $A_i=A\cap\sigma_i$ is super dense in $\sigma_i$ whenever $A_i\ne\emptyset$.}
\end{definition}

Observe that $A$ is nearly closed if and only if $\overline{A}^{\bullet}=A$. Notice also that every super dense set is p-dense.

\begin{theorem}\cite[Theorem 1]{KaMatStud:2015}\label{th:preimage_of_ssc}
Let $(X_t)_{t\in T}$ be a family of topological spaces, $a\in X_T$,  $\tau$ be a topology on $\sigma(a)$ and $f:\sigma(a)\to \mathbb R$ be an ssc-function. Then $f^{-1}(V)$ in nearly open in $\sigma(a)$ for any open set $V\subseteq \mathbb R$.
\end{theorem}

\begin{proposition}\label{prop:superdense_is_determ}
  Let $(X_t:t\in T)$ be a family of topological spaces and $E$ be a super dense set in an $S$-open space $X\subseteq X_T$ equipped with a locally projectively symmetric  topology $\tau$. Then $E$ is determining set for the class ${\rm SSC}(X)$.
\end{proposition}

\begin{proof}
  Consider ssc-functions $f,g:X\to \mathbb R$ such that $f|_E=g|_E$ and let $(\sigma_i:i\in I)$ be a partition of $X$ on mutually disjoint $\mathcal S$-components.  Denote
  $$
  H=\{x\in X:f(x)-g(x)=0\}
  $$
  and for every $i\in I$ we put $E_i=E\cap\sigma_i$ and $H_i=H\cap \sigma_i$. Since the function $h(x)=f(x)-g(x)$ is strongly separately continuous on $X$ by Theorem~\ref{th:operations-ssc}, every set $H_i=(h^{-1}(0))\cap \sigma_i$ is nearly closed in $\sigma_i$ according to Theorem~\ref{th:preimage_of_ssc}. Then the inclusion $E_i\subseteq H_i$ implies that $\sigma_i=\overline{E_i}^{\bullet}\subseteq\overline{H_i}^{\bullet}=H_i$. Consequently, $H_i=\sigma_i$ for every $i\in I$. Therefore, $H=\bigcup\limits_{i\in I} H_i=\bigcup\limits_{i\in I}\sigma_i=X$.
\end{proof}

\begin{lemma}\label{lem:loc_comp}
  Let $(X_n)_{n=1}^\infty$ be a sequence of locally compact Hausdorff spaces, $a\in\prod\limits_{n=1}^\infty X_n$ and $W\subseteq \sigma(a)$ be a nearly open set. Then for every $x\in W$ there exists a sequence $(U_n)_{n=1}^\infty$ of functionally open sets $U_n\subseteq X_n$ such that $x\in\bigl(\prod\limits_{n=1}^\infty U_n\bigr)\cap\sigma(a)\subseteq W$.
\end{lemma}

\begin{proof}
For every $n\in\mathbb N$ we denote $Y_n=\prod\limits_{k=1}^n X_k$ and $W_n=W_{\{1,\dots,n\}}=\{x\in Y_n: a_{\{1,\dots,n\}}^x\in W\}$. Take a number $N$ such that $x_n=a_n$ for all $n>N$.
Notice that $W_n$ is an open subset of the locally compact Hausdorff (and, consequently, completely regular) space $Y_n$. Therefore, for every $n=1,\dots,N$ there exists a functionally open neighborhood $U_n$ of $x_n$ with compact closure such that
$K_1=\prod\limits_{n=1}^N \overline{U_n}\subseteq W_N$. For every $x\in K_1$ we take a functionally open neighborhood $V_x\times G_x$ of $(x_1,
\dots,x_N,a_{N+1})$ with compact closure in $Y_N\times X_{N+1}$ such that $(x_1,
\dots,x_N,a_{N+1})\in V_x\times G_x\subseteq W_{N+1}$. Since the set $K_1\times \{a_{N+1}\}$ is compact, there exists a finite set $I\subseteq K_1$ such that $K_1\times\{a_{N+1}\}\subseteq \bigcup\limits_{x\in I} (V_x\times G_x)$. Put $U_{N+1}=\bigcap\limits_{x\in I} G_x$. Then $U_{N+1}$ is a functionally open neighborhood of $a_{N+1}$ in $X_{N+1}$ and $K_2=K_1\times\overline{U_{N+1}}\subseteq W_{N+1}$. Proceeding inductively in this way we obtain a sequence $(U_n)_{n=1}^\infty$ of functionally open sets $U_n\subseteq X_n$ with $x\in\bigl(\prod\limits_{n=1}^\infty U_n\bigr)\cap\sigma(a)\subseteq W$.
\end{proof}

The following result follows from \cite[Lemma 2]{KaMatStud:2015}.
\begin{proposition}\label{prop:base_set}
  Let $(X_n)_{n=1}^\infty$ be a sequence of topological  spaces, $a\in \prod\limits_{n=1}^\infty X_n$, $(U_n)_{n=1}^\infty$ be a sequence of functionally open sets $U_n\subseteq X_n$,  $W=\bigl(\prod\limits_{n=1}^\infty U_n\bigr)\cap\sigma(a)$ and let $\tau$ be the topology of pointwise convergence on $\sigma(a)$. Then there exists an ssc-function $f:(\sigma(a),\tau)\to [0,1]$ such that $W=f^{-1}((0,1])$.
\end{proposition}

\begin{proposition}\label{prop:determ_is_superdense_count}
  Let $X\subseteq \prod\limits_{n=1}^\infty X_n$ be an $\mathcal S$-open subset of the product of  a sequence of locally compact Hausdorff spaces $X_n$, $\mathcal T$ is a topology on $X$ which is finer than the topology $\tau$ of pointwise convergence and  $E\subseteq X$ be a determining set for the class  ${\rm SSC}(X)$. Then $E$ is super dense in $X$.
\end{proposition}

\begin{proof} Consider a partition $(\sigma_i:i\in I)$ of $X$ on mutually disjoint $\mathcal S$-open components. Assume that $E$ is not super dense in $X$. Then there exists $i\in I$ such that $\emptyset\ne \overline{E\cap\sigma_i}^{\bullet}\ne\sigma_i$. Since $W=\sigma_i\setminus \overline{E\cap\sigma_i}^{\bullet}$ is a nonempty nearly open set in $\sigma_i$, by Lemma \ref{lem:loc_comp} there exists a sequence $(U_n)_{n=1}^\infty$ of functionally open sets $U_n\subseteq X_n$ such that $G=\bigl(\prod\limits_{n=1}^\infty U_n\bigr)\cap\sigma_i\subseteq W$. According to Proposition \ref{prop:base_set} there exists an ssc-function $f:(\sigma_i,\tau)\to [0,1]$ such that $G=f^{-1}((0,1])$. Since $\tau\subseteq\mathcal T$, $f$ is strongly separately continuous on $(\sigma_i,\mathcal T)$. Notice that $f|_{E_\cap\sigma_i}=0$ and $f(x)>0$ for every $x\in G$, which implies a contradiction, since $E\cap \sigma_i$ is determining in $\sigma_i$.
\end{proof}

Since every space $\ell_p$ is an $S$-open subset of a countable product $\mathbb R^{\omega}$ and the standard topology on $\ell_p$ is finer than the topology of pointwise convergence, Propositions~\ref{prop:superdense_is_determ} and \ref{prop:determ_is_superdense_count} immediately imply the following result.

\begin{theorem}
For $p\in [1,+\infty)$ a set $E\subseteq\ell_p$ is determining for the class ${\rm SSC}(\ell_p)$ if and only if $E$ is super dense in $\ell_p$.
\end{theorem}

\section{Baire classification of ssc-functions on $\ell_p$}
For $p\in [1,+\infty)$ and $x=(x_n)_{n=1}^\infty,y\in\ell_p$ we denote
\begin{gather*}
\|x\|_p=\Bigl(\sum\limits_{n=1}^\infty |x_n|^p\Bigr)^{\frac 1p}\quad\mbox{and}\quad d_p(x,y)=\|x-y\|_p.
\end{gather*}

Let $\pi_n(x)=x_n$ for every $x=(x_n)_{n=1}^\infty\in\ell_p$ and $n\in\mathbb N$.

\begin{lemma}\label{lem:univ_alpha}
  For any $\alpha\in [1,\omega_1)$ and $p\in[1,+\infty)$ there exists an $\mathcal S$-open set $E$ in $\ell_p$ such that $E$ belongs to the $\alpha$'th additive class and does not belong to the $\alpha$'th multiplicative class.
\end{lemma}

\begin{proof} Fix $p\in[1,+\infty)$.

We define inductively sequences $(\tilde A_\alpha)_{1\le\alpha<\omega_1}$ and $(\tilde B_\alpha)_{1\le\alpha<\omega_1}$ of subsets of $\mathbb R^\omega$ in the following way. Put
\begin{gather*}
\tilde A_1=\{x=(x_n)_{n=1}^\infty\in \mathbb R^\omega: \exists m\,\,\forall n\ge m\,\,\, x_n=0\}\quad\mbox{and}\quad B_1=\mathbb R^\omega\setminus A_1.
\end{gather*}
Let $\mathbb N=\bigcup\limits_{n=1}^\infty T_n$ be a union of a sequence of mutually disjoint infinite sets
$T_n=\{t_{n1},t_{n2},\dots\}$, where $(t_{nm})_{m=1}^\infty$ is a strictly increasing sequence of numbers $t_{nm}\in\mathbb N$.
For every $n\in\mathbb N$ we denote by $\tilde A_1^n$ /$\tilde B_1^n$/ the copy of the set $\tilde A_1$ /$\tilde B_1$/, which is contained in the
space $\mathbb R^{T_n}$.
Assume that for some $\alpha\ge 1$ we have  already defined sequences $(\tilde A_\beta)_{1\le\beta<\alpha}$ and  $(\tilde B_\beta)_{1\le\beta<\alpha}$ (and their copies $(\tilde A_\beta^n)_{1\le\beta<\alpha}$ and  $(\tilde B_\beta^n)_{1\le\beta<\alpha}$ in $\mathbb R^{T_n}$) of subsets of $\mathbb R^\omega$. Now we put
\begin{gather*}
 \tilde A_\alpha=\left\{\begin{array}{ll}
                    \bigcup\limits_{m=1}^\infty \bigcap\limits_{n=m}^\infty \pi_{T_n}^{-1}(\tilde B_\beta^n), & \alpha=\beta+1, \\
                    \bigcup\limits_{n=1}^\infty \pi_{T_n}^{-1}(\tilde A_{\beta_n}^n), & \alpha=\sup\beta_n,
                  \end{array}
  \right.\\
  \tilde B_\alpha=\mathbb R^\omega\setminus \tilde A_\alpha.\phantom{aaaaaaaaaaaaaaaaaaaaaaa}
\end{gather*}
Let for every $\alpha\in [1,\omega_1)$
\begin{gather*}
  A_\alpha=\tilde A_\alpha\cap\ell_p, \,\,\, B_\alpha=\tilde B_\alpha\cap\ell_p,\\
  A_\alpha^n=\tilde A_\alpha^n\cap\ell_p(\mathbb R^{T_n}), \,\,\, B_\alpha^n=\tilde B_\alpha\cap\ell_p(\mathbb R^{T_n}).
\end{gather*}
{\sc Claim 1.} {\it For every $\alpha\in[1,\omega_1)$ the sets $A_\alpha$ and $B_\alpha$ are $\mathcal S$-open in $\ell_p$.}

{\it Proof of Claim 1.} Evidently, $A_1$ and $B_1$ are $\mathcal S$-open. Assume that for some $\alpha<\omega_1$ the claim is valid for all $\beta<\alpha$. Let $\alpha=\beta+1$ be an isolated ordinal. Take any $x\in A_\alpha$ and $y\in\sigma_1(x)$. Then  there exists $m\in\mathbb N$ such that $\pi_{T_n}(x)\in B_\beta^n$ for all $n\ge m$. Since $\pi_{T_n}(y)\in\sigma_1(\pi_{T_n}(x))$  and $B_\beta^n$ is $\mathcal S$-open, $\pi_{T_n}(y)\in B_\beta^n$. Therefore, $y\in A_\alpha$. We argue similarly in the case where $\alpha$ is a limit ordinal.\hfill$\Box$

Consider the equivalent metric
$$
d(x,y)=\min\{d_p(x,y),1\}
$$
on the space $\ell_p$.

{\sc Claim 2.} {\it For every $\alpha\in [1,\omega_1)$ the following condition holds:
\begin{itemize}
  \item[$(*)$] for every set $C\subseteq (\ell_p,d)$ of the additive /multiplicative/ class $\alpha$ there exists a contracting mapping $f:(\ell_p,d)\to (\ell_p,d)$ with the Lipschitz constant $q=\frac 12$ such that
      \begin{gather}
        C=f^{-1}(A_\alpha) \quad /C=f^{-1}(B_\alpha)/,\\
        |\pi_n(f(x))|\le 1\quad \forall x\in\ell_p\,\,\,\forall n\in\mathbb N.
      \end{gather}
\end{itemize}}

{\it Proof of Claim 2.}  We will argue by the induction on $\alpha$. Let $\alpha=1$ and $C$ be an arbitrary $F_\sigma$-subset of $(\ell_p,d)$. Then $C=\bigcup\limits_{n=1}^\infty C_n$ is a union of an increasing sequence of closed sets $C_n\subseteq (\ell_p,d)$. For every $x\in\ell_p$ we put
$$
f(x)=\bigl(\frac 13 d(x,C_1),\dots,\frac{1}{3^n}d(x,C_n),\dots\bigr).
$$
Since every $d(x,C_n)\le 1$, $f(x)\in\ell_p$ for every $x\in\ell_p$. Show that $C=f^{-1}(A_1)$. Take $x\in C$ and choose $m\in\mathbb N$ such that $x\in C_n$ for all $n\ge m$. Then $d(x,C_n)=0$ and $\pi_n(f(x))=0$ for all $n\ge m$. Hence, $x\in f^{-1}(A_1)$. The inverse inclusion follows from the closedness of $C_n$. Since
$$
d(f(x),f(y))\le d_p(f(x),f(y))=\Bigl(\sum\limits_{n=1}^\infty \frac{1}{3^{np}}|d(x,C_n)-d(y,C_n)|^p\Bigr)^{\frac 1p}\le d(x,y)\Bigl(\sum\limits_{n=1}^\infty \frac{1}{3^{np}}\Bigr)^{\frac 1p}\le\frac 12 d(x,y)
$$
for all $x,y\in\ell_p$, the mapping $f:(\ell_p,d)\to (\ell_p,d)$ is contracting with the Lipschitz constant $q=\frac 12$. Moreover, $|\pi_n(f(x))|=\frac{1}{3^n}d(x,C_n)\le 1$ for every $n\in\mathbb N$.

Assume that for some $\alpha<\omega_1$ the condition $(*)$ is valid for all $\beta<\alpha$. Let $C\subseteq (\ell_p,d)$ be any set of the $\alpha$'th additive class. Take an increasing sequence of sets $C_n$ such that $C=\bigcup\limits_{n=1}^\infty C_n$, where  every $C_n$ belongs to the multiplicative class $\beta$ if $\alpha=\beta+1$,  and in the case $\alpha=\sup\beta_n$ we can assume that $C_n$ belongs to the additive class $\beta_n$ for every $n\in\mathbb N$.  By the inductive assumption for every $n\in\mathbb N$ there exists a contracting mapping $f_n:(\ell_p,d)\to (\ell_{p},d)$ with the Lipschitz constant $q=\frac 12$ such that
\begin{gather}
  C_n=\left\{\begin{array}{ll}
                f_n^{-1}(B_\beta),  & \alpha=\beta+1,\\
                f_n^{-1}(A_{\beta_n}), & \alpha=\sup\beta_n,
             \end{array}\right.\\
  |\pi_m(f_n(x))|=|f_{nm}(x)|\le 1\quad\forall x\in\ell_p\,\,\,\forall n,m\in\mathbb N.
\end{gather}

For every $k\in\mathbb N$ we choose a unique pair $(n(k),m(k))\in\mathbb N^2$ such that
$$
k=t_{n(k)m(k)}\in T_{n(k)}.
$$
For every $x\in\ell_p$ we put
$$
f(x)=\bigl(\frac{1}{3^2} f_{n(1)m(1)}(x),\dots,\frac{1}{3^{k+1}} f_{n(k)m(k)}(x),\dots \bigr).
$$
It is easy to see that $f(x)\in \ell_p$ for every $x\in\ell_p$.

Since
\begin{gather*}
  \frac{1}{3}|f_{nm}(x)-f_{nm}(y)|\le d_p(f_n(x),f_n(y))
 \quad\mbox{and}\quad \frac{1}{3}|f_{nm}(x)-f_{nm}(y)|\le \frac 23\le 1,
\end{gather*}
we have
\begin{gather*}
\frac{1}{3}|f_{nm}(x)-f_{nm}(y)|\le d(f_n(x),f_n(y))\le \frac 12 d(x,y)
\end{gather*}
for all $x,y\in\ell_p$ and $n,m\in\mathbb N$.
Consequently,
\begin{gather*}
  d(f(x),f(y))\le d_p(f(x),f(y))=\Bigl(\sum\limits_{k=1}^\infty \frac{1}{3^{kp}}\Bigl(\frac 13|f_{n(k)m(k)}(x)-f_{n(k)m(k)}(y)|\Bigr)^p\Bigr)^{\frac 1p}\le \\
  \le\frac 12 d(x,y)\Bigl(\sum\limits_{k=1}^\infty \frac{1}{3^{kp}}\Bigr)^{\frac 1p}\le\frac 12d(x,y)
\end{gather*}
for all $x,y\in\ell_p$. Therefore, $f$ has the Lipshitz constant $q=\frac 12$.

Finally, it is easy to verify that $C=f^{-1}(A_\alpha)$.\hfill$\Box$

{\sc Claim 3.} {\it For every $\alpha\in [1,\omega_1)$ the set $A_\alpha$  belongs to the additive class $\alpha$  and does not belong to the multiplicative class $\alpha$ in $\ell_p$.}

{\it Proof of  Claim 3.} We first prove that the set $\tilde A_\alpha$ /$\tilde B_\alpha$/ is of the $\alpha$'th additive /multiplicative/ class in $\mathbb R^\omega$.

If $\alpha=1$, then $\tilde A_1=\bigcup\limits_{n=1}^\infty \sigma_n(0)$ is an $F_\sigma$-subset of $\mathbb R^\omega$, since every $\sigma_n(0)$ is closed in $\mathbb R^\omega$. Consequently, $\tilde B_1$ is $G_\delta$ in $\mathbb R^\omega$. Suppose that for some $\alpha\ge 1$ the set $\tilde A_\beta$ /$\tilde B_\beta$/ belongs to the additive /multiplicative/ class $\beta$ in $\mathbb R^\omega$ for every $\beta<\alpha$. Since every projection $\pi_{T_n}:\mathbb R^\omega\to\mathbb R^{T_n}$ is continuous, the set $\tilde A_\alpha$ belongs to the additive class $\alpha$ in $\mathbb R^\omega$ and the set $\tilde B_\beta$ belongs to the multiplicative class $\alpha$ in $\mathbb R^\omega$.

Since the topology of pointwise convergence on $\ell_p$ is weaker than the topology generated by the norm $\|\cdot\|_p$, for every $\alpha$ the set $A_\alpha$ /$B_\alpha$/ is of the $\alpha$'th additive /multiplicative/ class in $\ell_p$.

Fix $\alpha\in[1,\omega_1)$. In order to show that $A_\alpha$ does not belong to the $\alpha$'th multiplicative class we assume the contrary. Then there exists a contracting mapping $f:(\ell_p,d)\to (\ell_p,d)$ such that $A_\alpha=f^{-1}(B_\alpha)$. By the Contraction Map Principle, there exists a fixed point for the mapping $f$, which implies a contradiction.

It remains to put $E=A_\alpha$.
\end{proof}

\begin{theorem}\label{th:Baire_ssc}
  Let $\alpha\in[1,\omega_1)$ and $p\in[1,+\infty)$. Then there exists an ssc-function $f:\mathbb \ell_p\to\mathbb R$ which belongs to the $(\alpha+1)$'th Baire class and does not belong to the $\alpha$'th Baire class.
\end{theorem}

\begin{proof}
By Lemma~\ref{lem:univ_alpha} there exists an $\mathcal S$-open set $E$ in $\ell_p$ such that $E$ belongs to the $\alpha$'th additive class and does not belong to the $\alpha$'th multiplicative class. Then the function $f:\ell_p\to\mathbb R$,
  \begin{gather*}
    f(x)=\chi_{E}(x)=\Bigl\{\begin{array}{ll}
                              1, & x\in E, \\
                              0, & x\not\in E,
                            \end{array}
  \end{gather*}
satisfies the required properties.
\end{proof}

The existence of  an ssc-function $f:\mathbb R^\omega\to\mathbb R$ which is not Baire measurable  was proved in \cite[Proposition 3.2]{KaRAEX:2015}.

\section{Discontinuities of ssc-functions on $\ell_p$}\label{sec:discont}
We denote by $\ell_p^\omega$ the set $\ell_p\subseteq\mathbb R^\omega$ endowed with the topology of pointwise convergence induced from $\mathbb R^\omega$. Evidently, ${\rm SSC}(\ell_p^\omega)\subseteq {\rm SSC}(\ell_p)$. The converse is not true as the following example shows.

\begin{example}
  There exists a continuous function $f:\ell_2\to\mathbb R$ which is not strongly separately continuous on $\ell_2^\omega$.
\end{example}

\begin{proof}
  For every $n\in\mathbb N$ we set
  \begin{gather*}
  x_n=(\mathop{\underbrace{\frac 1n,0,\dots,0}}_{n-1},n,0,\dots)\,\,\,\mbox{and}\,\,\,   y_n=(\mathop{\underbrace{0,\dots,0}}_{n-1},n,0,\dots).
  \end{gather*}
  Since the sets $F_1=\{x_n:n\in\mathbb N\}$ and $F_2=\{y_n:n\in\mathbb N\}$ are disjoint and closed in $\ell_2$, the function $f:\ell_2\to\mathbb R$ defined by the formula
  \begin{gather*}
    f(x)=\frac{d_2(x,F_1)}{d_2(x,F_1)+d_2(x,F_2)}
  \end{gather*}
  is continuous and $F_1=f^{-1}(0)$, $F_2=f^{-1}(1)$. Notice that $x_n\to 0$ in $\ell_2^\omega$ and $y_n=(x_n)_1^0$. But $f(x_n)-f(y_n)=1$ for every $n$, which implies that $f$ is not strongly separately continuous on $\ell_2^\omega$ at the point $x=0$ with respect to the first variable.
\end{proof}

By $C(f)$ ($D(f)$) we denote the set of all points of continuity (discontinuity) of a mapping $f$.

\begin{theorem}
 For any open nonempty set $G\subseteq \ell_p$ with $1\le p<\infty$ there exists a strongly separately continuous function $f:\ell_p\to \mathbb R$ such that $D(f)=G$.
\end{theorem}

\begin{proof} Fix $p\in [1,+\infty)$ and let $\sigma=\sigma(0)=\{(x_n)_{n=1}^\infty\in\ell_p: (\exists k\in\mathbb N)\,\,\, (\forall n\ge k)\,\,\,(x_n=0)\}$. Denote $F=\ell_p\setminus G$. For every $x=(x_n)_{n\in\mathbb N}\in \ell_p$ we put
\begin{gather*}
   \varphi(x)=\left\{\begin{array}{ll}
                       \min\{d_p(x,F),1\}, &   F\ne\emptyset,\\
                       1, &  F=\emptyset,
                     \end{array}
   \right.\\
g(x)=\left\{\begin{array}{ll}
                       \exp(-\sum\limits_{n=1}^\infty |x_n|), &  x=(x_n)_{n=1}^\infty\in \sigma,\\
                       1, &  x\in \ell_p\setminus \sigma,
                     \end{array}
   \right.
\end{gather*}
and let
\begin{gather*}
      f(x)=\varphi(x)\cdot g(x).
\end{gather*}

{\sc Claim 1.}  {\it $F\subseteq C(f)$.}

{\it Proof.} Fix $x^0\in F$ and take any convergent sequence $(x^m)_{m=1}^\infty$ to $x^0$ in $\ell_p$.  Notice that
\begin{gather*}
  \lim\limits_{m\to\infty}\varphi(x^m)g(x^m)=0,
\end{gather*}
because $\varphi(x)$ is continuous at $x^0$ and $g(x)$ is bounded. Then
\begin{gather*}
  \lim\limits_{m\to\infty}f(x^m)=0=f(x^0).
\end{gather*}
Hence, $x^0\in C(f)$.

{\sc Claim 2.} {\it $G\subseteq D(f)$.}

{\it Proof.} Fix $x^0\in G$. Then $f(x^0)>0$. We put $\varepsilon=\frac 12 f(x^0)$ and take an arbitrary $\delta>0$. Since the set $D=\ell_p\setminus\sigma$ is dense in $\ell_p$, there exists $x=(x_n)_{n\in\mathbb N}\in\ell_p$ such that
\begin{gather*}
  \|x-x^0\|_p<\frac{\delta}{2}\,\,\,\mbox{and}\,\,\, x\not\in\sigma.
\end{gather*}
Take a number $N$ such that
\begin{gather*}
  \sum\limits_{n=1}^N |x_n|>\ln\Bigl(\frac{1}{f(x^0)-\varepsilon}\Bigr)\,\,\,\,\mbox{and}\,\,\,\sum\limits_{n=N+1}^\infty |x_n|^p<\Bigl(\frac{\delta}{2}\Bigr)^p.
\end{gather*}
We put
$$
y=(x_1,\dots,x_N,0,0,\dots).
$$
Then $y\in\sigma$ and
\begin{gather*}
\|y-x^0\|_p\le \|y-x\|_p+\|x-x^0\|_p=\Bigl(\sum\limits_{n=N+1}^\infty |x_n|^p\Bigr)^{\frac 1p}+\|x-x^0\|_p<\frac\delta 2+\frac\delta 2=\delta.
\end{gather*}
But
\begin{gather*}
  f(x^0)-f(y)=f(x^0)-\varphi(y)\cdot \exp\bigl(-\sum\limits_{n=1}^N |x_n|\bigr)>f(x^0)- \exp\bigl(-\sum\limits_{n=1}^N |x_n|\bigr)>f(x^0)+\varepsilon-f(x^0)=\varepsilon,
\end{gather*}
which implies that $f$ is discontinuous at $x^0$.

{\sc Claim 3.} {\it $f:\ell_p\to\mathbb R$ is strongly separately continuous.}

{\it Proof.} Let $x^0\in\ell_p$. Evidently, $f$ is strongly separately continuous at $x^0$ if  $x^0\in F$. Therefore, we assume that $x^0\in G$.
Fix $k\in\mathbb N$  and $\varepsilon>0$. Take $\delta_1>0$ with $B(x^0,\delta_1)\subseteq G$. Since $\varphi(x)$ is continuous at $x^0$, there exists $\delta_2>0$ such that
$$
|\varphi(x)-\varphi(x^0)|<\frac{\varepsilon}{4}
$$
for all $x\in B(x^0,\delta_2)$. Put
$$
\delta=\min\{\delta_1,\delta_2,\ln \bigl(1+\frac\varepsilon 2\bigr)\}.
$$
Now let $x\in B(x^0,\delta)$ and $y=x_k^{x^0}\in B(x^0,\delta)$. If $x\not\in\sigma$, then $y\not\in\sigma$. In this case
$$
|f(x)-f(y)|=|\varphi(x)-\varphi(y)|\le |\varphi(x)-\varphi(x^0)|+|\varphi(x^0)-\varphi(y)|<\varepsilon.
$$
Assume that $x\in \sigma$. Then $y\in\sigma$ and
\begin{gather*}
  |f(x)-f(y)|\le |g(x)||\varphi(x)-\varphi(y)|+|\varphi(y)||g(x)-g(y)|<\\
  <\frac{\varepsilon}{2}+|\exp(-\sum\limits_{n=1}^\infty |x_n|)-\exp(-\sum\limits_{n=1}^\infty |y_n|)|<\frac{\varepsilon}{2}+|\exp(\sum\limits_{n=1}^\infty |y_n|-\sum\limits_{n=1}^\infty |x_n|)-1|.
\end{gather*}
Taking into account that
$$
\exp(-\sum\limits_{n=1}^\infty |x_n-y_n|)\le \exp(\sum\limits_{n=1}^\infty |x_n|-\sum\limits_{n=1}^\infty |y_n|)\le \exp(\sum\limits_{n=1}^\infty |x_n-y_n|),
$$
we obtain that
\begin{gather*}
  |f(x)-f(y)|<\frac{\varepsilon}{2}+\exp(|x_k-x_k^0|)-1<\frac{\varepsilon}{2}+\exp(\delta)-1<\frac{\varepsilon}{2}+\frac{\varepsilon}{2}=\varepsilon.
\end{gather*}
Hence, $f$ is strongly separately continuous at $x^0$ with respect to the $k$'th variable.
\end{proof}

\end{document}